\documentclass[12pt]{article}
\usepackage{amsmath,amssymb,amsthm, amsfonts}
\usepackage{hyperref}
\usepackage{graphicx}
\usepackage{array, tabulary}
\usepackage{url}
\usepackage[mathlines]{lineno}
\usepackage{dsfont} 
\usepackage{color}
\usepackage{subcaption}
\usepackage{enumitem}
\definecolor{red}{rgb}{1,0,0}

\definecolor{blue}{rgb}{0,0,1}

\definecolor{green}{rgb}{0,.6,0}

\usepackage{float}
\usepackage[normalem]{ulem}
\usepackage{diagbox} 
\usepackage{tensor} 

\usepackage{tikz}
\usetikzlibrary{calc}

\newtheorem{thm}{Theorem}[section]
\newtheorem{cor}[thm]{Corollary}
\newtheorem{lem}[thm]{Lemma}
\newtheorem{prop}[thm]{Proposition}

\theoremstyle{definition}

\theoremstyle{definition}

\theoremstyle{definition}

\renewcommand{\vec}[1]{\mathbf{#1}}

\newcommand{\nul}{\operatorname{null}}

\newcommand{\dist}{\operatorname{dist}}

\newcommand{\bit}{\begin{itemize}}
\newcommand{\eit}{\end{itemize}}
\newcommand{\ben}{\begin{enumerate}}
\newcommand{\een}{\end{enumerate}}
\newcommand{\beq}{\begin{equation}}
\newcommand{\eeq}{\end{equation}}
\newcommand{\bea}{\begin{eqnarray}} 
\newcommand{\eea}{\end{eqnarray}}
\newcommand{\bpf}{\begin{proof}}
\newcommand{\epf}{\end{proof}\ms}
\newcommand{\bmt}{\begin{bmatrix}}
\newcommand{\emt}{\end{bmatrix}}
\newcommand{\ms}{\medskip}

\newcommand{\beqs}{\begin{equation*}} 
\newcommand{\eeqs}{\end{equation*}}
\newcommand{\beas}{\begin{eqnarray*}}
\newcommand{\eeas}{\end{eqnarray*}}

\newcommand{\spa}{\operatorname{span}}

\newcommand{\nonback}{B}
\newcommand{\glue}[2]{\hspace{-3pt}\tensor[_#1]{\circ}{_#2}\hspace{-2pt}}

\tikzstyle{vertex}=[draw,thick,fill=white,circle,inner sep=2pt]

\title{Defective eigenvalues of the non-backtracking matrix}
\author{
\and Kristin Heysse
\thanks{Department of Mathematics, Statistics, and Computer Science, Macalester College, St.\ Paul, MN, (kheysse@macalester.edu)}
\and Kate Lorenzen
\thanks{Department of Mathematics and Computer Science, Linfield University, McMinnville, OR, USA (klorenzen@linfield.edu)}
\and Carolyn Reinhart
\thanks{Department of Mathematics and Statistics, Swarthmore College, Swarthmore, PA, USA (creinha1@swarthmore.edu)}}

\date{\today}

\begin{document}

\maketitle
\begin{abstract}
    We consider graphs for which the non-backtracking matrix has defective eigenvalues, or graphs for which the matrix does not have a full set of eigenvectors. The existence of these values results in Jordan blocks of size greater than one, which we call nontrivial. We show a relationship between the eigenspaces of the non-backtracking matrix and the eigenspaces of a smaller matrix, completely classifying their differences among graphs with at most one cycle. Finally, we provide several constructions of infinite graph families that have nontrivial Jordan blocks for both this smaller matrix and the non-backtracking matrix.  
\end{abstract}

\section{Introduction}

A non-backtracking walk in a graph is any traversal of the vertices of a graph such that no edge is immediately repeated. The non-backtracking matrix encodes if two edges can be traversed in succession in a non-backtracking walk.
This matrix was originally introduced by Hashimoto in 1989 \cite{Hashimoto}. It has been used to study percolation \cite{Caletal}, community detection \cite{Betal}, and non-recurrent epidemic spread \cite{PSetal}. 

Of particular interest is the eigen-information of the non-backtracking matrix. The largest eigenvalue is related to the epidemic threshold of the SIR model \cite{PSetal}, while the eigenvectors have been used to rank importance of nodes in networks \cite{Martinetal}. Non-backtracking walks are also a better tool for community detection by spectral clustering in sparse networks. Developers of these spectral clustering algorithms commend the spectrum of the non-backtracking matrix for its ability to maintain a large gap between bulk eigenvalues and the eigenvalues related to community detection \cite{clustering}.

The non-backtracking matrix is not symmetric, making it one of the only well-studied graph matrices where there may not be a full set of linearly independent eigenvectors.
When the algebraic and geometric multiplicity of an eigenvalue are not equal, the eigenvalue is called \emph{defective} and the matrix is not diagonalizable. Graphs that have a defective eigenvalue we will call \emph{defective graphs}.
We look to answer the question, \textit{which graphs are defective for the non-backtracking matrix?}
Torres showed in \cite{Torres} that if a graph has a vertex of degree one, then the non-backtracking matrix of the graph will have a defective eigenvalue of $\lambda=0$. As such, the question can be rephrased: \textit{which graphs of minimum degree two are defective for the non-backtracking matrix?} It was conjectured by Torres in \cite{Torres} that the answer is never.

We show that this conjecture is false by providing a constructive method to build three infinite graph families with a defective eigenvalue (see Section \ref{sec: fams}). In some cases, this defective eigenvalue is real. In addition, we provide some numerical data regarding the number of graphs on 10 or fewer vertices with defective eigenvalues and the eigenvalues for which they are defective. Our constructions use a previously studied matrix, $K$, which is defined in terms of the adjacency and degree diagonal matrices of a graph and that shares eigenvalues with $\nonback$. 

We solidify the relationship between the Jordan canonical forms of $K$ and the non-backtracking matrix by completely classifying when they are the same and describing what differences can occur (see Section \ref{sec: JCFK}). Further, we construct useful results regarding the structure of the generalized eigenvectors of the matrix $K$ for special graph structures in Section \ref{sec: structure for k}. These results are also applied to determine conditions on the generalized eigenvectors of graphs containing twin vertices. 

\subsection{Preliminaries}

For a vector $\vec{v}$, we will let $\vec{v}_i$ denote the $i$th entry of the vector. The all ones vector of dimension $n$ will be denoted $\vec{1_n}$ (or just $\vec{1}$ when the dimension is clear) and the identity matrix of dimension $n$ will be denoted $I_n$ (or just $I$ when the dimension is clear). The $i$th standard basis vector will be denoted $\vec{e_i}$. 

Let $(\lambda, \vec{u}^{(1)})$ be an eigenpair for a matrix $M$. Further, let $k$ be the largest integer such that the system $(M-\lambda I)\vec{u}^{(j+1)}=\vec{u}^{(j)}$ has a solution for all $j \in \{0,1,\ldots k-1\}$, where we define $\vec{u}^{(0)}$ to be the zero vector. It follows that $k$ is at least 1, since $\vec{u}^{(1)}$ is an eigenvector. This sequence of vectors $\vec{u}^{(1)},\vec{u}^{(2)},\ldots, \vec{u}^{(k)}$ is a \emph{Jordan chain of length $k$} of $M$, and each vector $\vec{u}^{(j)}$ is a \emph{generalized eigenvector of $M$} for eigenvalue $\lambda$. A full set of Jordan chains, when considered as columns of a matrix, serve as the similarity transformation for a matrix into its Jordan canonical form. 

The lengths of the Jordan chains for a matrix $M$ correspond to the sizes of its Jordan blocks. That is, $M$ has a Jordan chain of length $k$ for eigenvalue $\lambda$ if and only if the Jordan canonical form of $M$ contains a Jordan block of size $k$. Furthermore, the geometric multiplicity of $\lambda$ is the number of Jordan blocks corresponding to $\lambda$ and the algebraic multiplicity of $\lambda$ is the sum of the sizes of all Jordan blocks for $\lambda$. Thus, the existence of a Jordan chain of length $k\geq 2$ for $M$ is sufficient to show that $M$ does not have a full set of eigenvectors.

A \textit{(simple) graph} $G$ consists of a set of vertices $V(G)$ and a set of edges $E(G)$ such that each edge $e\in E(G)$ is a subset of two vertices. For ease of notation, $ij$ will be used to denote the edge $\{i,j\}$. We also use $i\sim j$ to denote that $i$ and $j$ are connected by an edge. The set of all vertices $j$ such that $ij\in E(G)$ is called the \emph{neighborhood} of the vertex $i$. The \textit{degree} of a vertex $i$, denoted $\deg(i)$, is the size of its neighborhood. All graphs we consider are \emph{connected}, or have walk between any two vertices. The length of the shortest walk between two vertices $i$ and $j$ is defined to be the \emph{distance} between them, denoted $\dist(i,j)$.

To define the non-backtracking matrix of a graph, it is useful to consider each edge $ij$ in a graph as a pair of directed edges $(i,j)$ and $(j,i)$. This is due to the nature of how non-backtracking walks are defined: traveling from $i$ to $j$ along the edge $ij$ will result in a different set of viable next edges for the walk than traveling from $j$ to $i$ along $ij$. As such, the \textit{non-backtracking matrix} $\nonback(G)$ is indexed by the directed edges of the graph $G$ and defined such that 
\begin{align*}
    \nonback_{(i,j),(k,m)}(G)= \begin{cases}
        1 & \text{ if } j=k \text{ and } i \neq m\\
        0 & \text{ otherwise}.
    \end{cases}
\end{align*} In other words, if we can travel along $ij$ and then immediately travel $km$ in a non-backtracking walk, the corresponding matrix entry is one. Otherwise, the matrix entry is zero. See Figure \ref{fig:pawex} for an example.

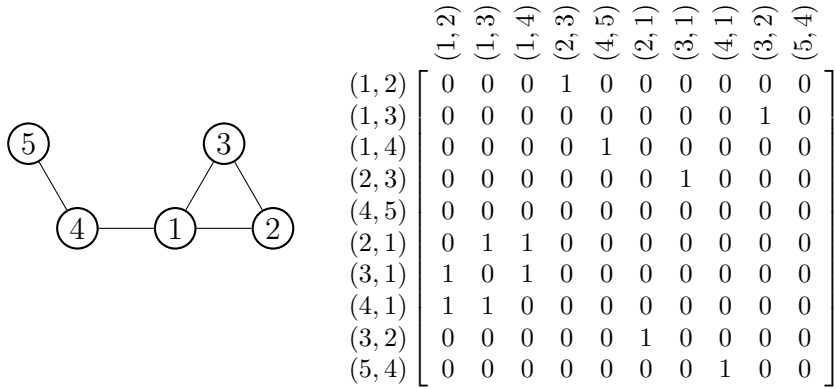
\begin{figure}[h!]
    \centering
\begin{tikzpicture}
\tikzstyle{knode}=[circle,draw=black,thick,inner sep=2pt]

\node (0) at (-1,0) [knode] {1};
\node (1) at ($(0)+(0:1.3 cm)$) [knode] {2};
\node (2) at ($(0)+(60:1.3 cm)$) [knode] {3}; 
\node (3) at ($(0)+(180: 1.3 cm)$) [knode] {4};
\node (4) at ($(3)+(120:1.3 cm)$) [knode] {5}; 

\foreach \from/\to in {0/1,1/2,2/0,0/3,3/4}
\draw (\from) -- (\to);

\node (m) at (5,0) {\footnotesize$ \left[ \begin{array}{rrrrrrrrrr}
0 & 0 & 0 & 1 & 0 & 0 & 0 & 0 & 0 & 0 \\
0 & 0 & 0 & 0 & 0 & 0 & 0 & 0 & 1 & 0 \\
0 & 0 & 0 & 0 & 1 & 0 & 0 & 0 & 0 & 0 \\
0 & 0 & 0 & 0 & 0 & 0 & 1 & 0 & 0 & 0 \\
0 & 0 & 0 & 0 & 0 & 0 & 0 & 0 & 0 & 0 \\
0 & 1 & 1 & 0 & 0 & 0 & 0 & 0 & 0 & 0 \\
1 & 0 & 1 & 0 & 0 & 0 & 0 & 0 & 0 & 0 \\
1 & 1 & 0 & 0 & 0 & 0 & 0 & 0 & 0 & 0 \\
0 & 0 & 0 & 0 & 0 & 1 & 0 & 0 & 0 & 0 \\
0 & 0 & 0 & 0 & 0 & 0 & 0 & 1 & 0 & 0
\end{array}\right]$};

\node (m) at (1.7,0) {\footnotesize $\begin{array}{r}
(1,2) \\ (1,3) \\ (1,4) \\ (2,3) \\ (4,5) \\ (2,1) \\ (3,1) \\ (4,1) \\ (3,2) \\ (5,4) \end{array}$};

\node (m) at (5,2.6) {\rotatebox{90}{\footnotesize $\begin{array}{r}
(1,2) \\[3pt] (1,3) \\[3pt] (1,4) \\[3pt] (2,3) \\[3pt] (4,5) \\[3pt]  (2,1) \\[3pt]  (3,1) \\[3pt]  (4,1) \\[3pt]   (3,2) \\[3pt] (5,4) \end{array}$}};

\end{tikzpicture}
    \caption{A graph on five vertices and its non-backtracking matrix. }
    \label{fig:pawex}
\end{figure}

The non-backtracking matrix has many eigenvalues and eigenvectors that are predictable as they correspond to cycles in a graph. Mainly, these are $\lambda=\pm1$. The spectral interest of this matrix tends to be on the remaining eigenvalues and corresponding eigenspaces. The non-backtracking matrix's spectrum can also be found using Ihara's Theorem which separates the predictable eigenvalues from the remaining ones. 

\begin{thm}[Ihara's Theorem \cite{Ihara}]
\label{thm:ihara}
For a graph $G$ with $n$ vertices and $m$ edges, let $\nonback$ be the non-backtracking matrix of $G$. Let $A$ be the adjacency matrix of $G$ and $D$ the degree diagonal matrix. Then
\begin{align*}
    \det(I-u\nonback)=(1-u^2)^{m-n}\det(u^2(D-I)-uA+I).
\end{align*}
\end{thm}

Note that this does not directly give us the eigenvalues of $\nonback$, since the eigenvalues are the solutions to $\det(\lambda I - \nonback)$. Therefore, for the solutions to $\det(I-u \nonback)$, it follows that $u=1/\lambda$. Moreover, the eigenvalues of $\nonback$ are $\pm 1$, each with algebraic multiplicity at least $m-n$, as well as the solutions to $\det(x^2I-xA+(D-I))=0$.

The remaining eigenvalues can be found by considering the following matrix $K$. For a graph $G$ with $n$ vertices and $m$ edges, the $2n \times 2n$ matrix $K(G)$ is defined as
\begin{align*}
    K=\begin{bmatrix}
    A & D-I_n\\
    -I_n & 0
\end{bmatrix},
\end{align*}
where $A$ and $D$ are as defined in Theorem \ref{thm:ihara}. Note that $\det(xI-K)=\det(x^2I-xA +(D-I))$. So all the eigenvalues of $K$ (respecting algebraic multiplicity) are eigenvalues of $\nonback$. This matrix $K$ has seen some interest for those working with the non-backtracking matrix because it is a smaller matrix ($2n \times 2n$ versus $2m \times 2m$) with alike spectral properties. 

Glover and Kempton in \cite{Kempton} further spectrally connected these two matrices by showing that $\nonback$ could be decomposed into a block diagonal matrix $M$ with $K$ as one of its blocks. For this decomposition, they defined the following matrices which we will use in Section \ref{sec: JCFK}. Let $S$ be a $2m \times n$ matrix and $T$ be an $n \times 2m$ matrix where 
\begin{center}
    \begin{tabular}{ccc}
       $S((u,v),x)= \begin{cases}
           1 & v=x\\
           0 & \text{ otherwise}
       \end{cases}$ & and &  $T(x,(u,v))= \begin{cases}
           1 & x=u\\
           0 & \text{ otherwise}
       \end{cases}$ .
    \end{tabular}
\end{center}
Let $R$ be a $2m \times 2(m-n)$ matrix, where the columns of $R$ are linearly independent and the eigenvectors of $\nonback$ for the eigenvalues $\pm 1$ which are in the null space of $ST$ (shown by Lubetzky and Peres in \cite{LP}). 

\begin{thm}[\cite{Kempton}]
    Let $G$ be a connected graph and $\nonback$ its non-backtracking matrix. Then
    \begin{align*}
        \nonback X = X \begin{bmatrix}
    K & 0 & 0\\
    0 & I_{m-n} & 0\\
    0 & 0 & -I_{m-n}
\end{bmatrix}
    \end{align*} and $X=\begin{bmatrix}
        S & T^T & R
    \end{bmatrix}$. 
\end{thm}

Let $M$ be the $2m \times 2m$ matrix
\begin{align*}
    M=\begin{bmatrix}
    K & 0 & 0\\
    0 & I_{m-n} & 0\\
    0 & 0 & -I_{m-n}
\end{bmatrix}.
\end{align*}

It is clear that $M$ has the same spectrum as $\nonback$. Glover and Kempton where able to use $M$ and $K$ to gain spectral information for $\nonback$. For example, they established that the geometric multiplicity of $\lambda=1$ of $K$ is the number of connected components of $G$ and related the spectral radius of $K$ and $A$ to that of $\nonback$.

\section{Jordan Canonical Form of $K$}\label{sec: JCFK}
Glover and Kempton show in \cite{Kempton} that eigenvectors for $K$ \emph{lift} to eigenvectors of $\nonback$, i.e.  if $\vec{v}$ is an eigenvector of $K$, then $X[\vec{v}, \vec{0}, \vec{0}]^T$ is an eigenvector for $\nonback$. First, we will generalize this result to Jordan chains.

\begin{prop}
\label{prop:EigenLift} 
Let $\vec{u}^{(1)},\dots,\vec{u}^{(k)}$ be a Jordan chain of $K$ for the eigenvalue $\lambda$ and let $\vec{v}^{(i)}=\begin{bmatrix}
    \vec{u}^{(i)} & 0 &0\end{bmatrix}^T$. If $X\vec{v}^{(1)}\not=0$, then $X\vec{v}^{(1)},\dots,X\vec{v}^{(k)}$ is a Jordan chain of $\nonback$ for eigenvalue $\lambda$. If $X\vec{v}^{{(1)}}=\dots=X\vec{v}^{{(i-1)}}=0$ and $X\vec{v}^{{(i)}}\not=0$ for some $1<i\leq k$, then then $X\vec{v}^{(i)},\dots,X\vec{v}^{(k)}$ is a Jordan chain of $\nonback$ for eigenvalue $\lambda$.
\end{prop}

\begin{proof}
Let $\vec{u}^{(1)}$ is an eigenvector of $K$ for $\lambda$, that is $K\vec{u}^{(1)}=\lambda\vec{u}^{(1)}$. Then

\begin{align*}   \nonback X\vec{v}^{(1)}=X\begin{bmatrix}
    K & 0 & 0\\
    0 & I & 0\\
    0 & 0 & -I
\end{bmatrix}\begin{bmatrix}
    \vec{u}^{(1)}\\
    0\\
    0
\end{bmatrix}=X\begin{bmatrix}
    K\vec{u}^{(1)}\\ 0\\ 0\end{bmatrix}=X \lambda \begin{bmatrix}
    \vec{u}^{(1)}\\
    0\\
    0
\end{bmatrix}=\lambda X \vec{v}^{(1)}.
\end{align*}

If $X\vec{v}^{(1)}\not=0$, then $X\vec{v}^{(1)}$ is an eigenvector of $\nonback$ for eigenvalue $\lambda$.

Let $\vec{u}^{(1)},\dots,\vec{u}^{(k)}$ be a Jordan chain of $K$ for $\lambda$. Therefore $K\vec{u}^{{(i)}}=\lambda \vec{u}^{{(i)}}+\vec{u}^{{(i-1)}}$ for $1<i\leq k$. Consider
\begin{align*}
    \nonback X\vec{v}^{(i)}=X\begin{bmatrix}
    K & 0 & 0\\
    0 & I & 0\\
    0 & 0 & -I
\end{bmatrix}\begin{bmatrix}
    \vec{u}^{(i)}\\
    0\\
    0
\end{bmatrix}=X\begin{bmatrix}
    K\vec{u}^{(i)}\\ 0\\ 0\end{bmatrix}=X \begin{bmatrix}
   \lambda \vec{u}^{(i)}+\vec{u}^{{(i-1)}}\\
    0\\
    0
\end{bmatrix}=\lambda X \vec{v}^{(i)}+X\vec{v}^{{(i-1)}}.
\end{align*}

 Further $\nonback X\vec{v}^{(i)}- \lambda X\vec{v}^{(i)}= X\vec{v}^{{(i-1)}}$.
If $X\vec{v}^{{(1)}}\not=0$, then $\nonback X\vec{v}^{(2)}-\lambda X\vec{v}^{(2)}\not=0$. So $X \vec{v}^{(2)}$ is not in the eigenspace of $\lambda$ nor the null vector. Therefore $X\vec{v}^{(2)}$ is a generalized eigenvector of $\nonback$ for eigenvalue $\lambda$. 

    For the sake of induction, assume $X\vec{v}^{(i-1)}\not=0$, so $\nonback X\vec{v}^{(i)}-\lambda X\vec{v}^{(i)}\not=0$. Therefore $X\vec{v}^{(i)}\not=0$, so $X\vec{v}^{(i)}$ is a generalized eigenvector of $\nonback$ for eigenvalue $\lambda$. Therefore, $X\vec{v}^{(1)},\dots,X\vec{v}^{(k)}$ is a Jordan chain of $\nonback$ for the eigenvalue $\lambda$.

If $X\vec{v}^{{(1)}}=\dots=X\vec{v}^{{(i-1)}}=0$ and $X\vec{v}^{{(i)}}\not=0$ for some $1<i\leq k$, then $\nonback X\vec{v}^{(i)}=\lambda X\vec{v}^{(i)}$. Thus, $X\vec{v}^{(i)}$ is an eigenvector of $\nonback$ for $\lambda$. By the same argument as in the previous case, $X\vec{v}^{({\ell})}\not=0$ for all $i+1\leq \ell\leq k$. Therefore, $X\vec{v}^{(i)},\dots,X\vec{v}^{(k)}$ is a Jordan chain of $\nonback$ for the eigenvalue $\lambda$.\end{proof}

We now investigate the null space of $\begin{bmatrix}
    S & T^T
\end{bmatrix}$. This will allow us to connect the eigenspace properties of $K$ and $\nonback$.

\begin{lem}\label{lemma: null space of transition matrix}
Let $G$ be a graph on $n$ vertices. If $G$ is not bipartite, $\nul \left(\begin{bmatrix} S & T^T \end{bmatrix}\right)=\spa\left\{\begin{bmatrix} \vec{1_n}  & -\vec{1_n} \end{bmatrix}^T\right\}$. If $G$ is a bipartite graph with partite sets $\mathcal{A}$ and $\mathcal{B}$, then $\nul \left(\begin{bmatrix} S & T^T \end{bmatrix}\right)=\spa\left\{\begin{bmatrix} \vec{1_{|\mathcal{A}|}}  & \vec{1_{|\mathcal{B}|}} & -\vec{1_{|\mathcal{A}|}} & -\vec{1_{|\mathcal{B}|}}\end{bmatrix}^T,\begin{bmatrix} \vec{1_{|\mathcal{A}|}}  & -\vec{1_{|\mathcal{B}|}} & \vec{1_{|\mathcal{A}|}} & -\vec{1_{|\mathcal{B}|}}\end{bmatrix}^T\right\}$.
\end{lem} 

\begin{proof}
Let $G$ be a simple, undirected graph on $n$ vertices. We will use the linear algebra fact that $\begin{bmatrix}
    S & T^T
\end{bmatrix}\vec{v}=\vec{0}$ if and only if $\begin{bmatrix}
    S^T \\T
\end{bmatrix}\begin{bmatrix}
    S & T^T
\end{bmatrix}\vec{v}=\vec{0}$.

Since Glover and Kempton \cite{Kempton} showed that $ST=A$, the adjacency matrix and by matrix multiplication $S^TS=TT^T=D$, the degree diagonal matrix, we see
$$    \begin{bmatrix}
    S^T \\T
\end{bmatrix}\begin{bmatrix}
    S & T^T
\end{bmatrix} = \begin{bmatrix}
    S^TS &S^T T^T\\
    ST & T T^T
\end{bmatrix}=\begin{bmatrix}
    D &A\\
    A & D
\end{bmatrix}.$$

Therefore, we will consider the null space of $\begin{bmatrix}
    D &A\\
    A & D
\end{bmatrix}$. Let $\vec{v}=[\vec{x}, \vec{y}]^T$ be in the null space of $\begin{bmatrix}
    D &A\\
    A & D
\end{bmatrix}$. So $D\vec{x} +A \vec{y}=\vec{0}$ and $A\vec{x}+D\vec{y}=\vec{0}$.
Since $D$ in invertible we get
$$\vec x = -D^{-1}A\vec y \quad \text{ and } \quad
    \vec y = -D^{-1}A\vec x.$$
Therefore,
\begin{align*}
    \vec x = (D^{-1}A)^2 \vec{x},
\end{align*}
that is that $\vec{x}$ is an eigenvector for the probability transition matrix (a stochastic matrix) of $G$ for the eigenvalue $\lambda=\pm 1$. For a simple, connected graph, we know the probability transition matrix has an eigenvalue $\lambda=1$ with multiplicity one always and eigenvalue $\lambda=-1$ with multiplicity one if and only if the graph is bipartite. 

Further, when $\lambda=1$, then $\vec{x}= \vec{1}$. Therefore $\vec{y}=-\vec{1}$. When $\lambda=-1$, then $\vec{x}=[\vec{1}, -\vec{1}]^T$. Therefore $\vec{y}=[-\vec{1}, \vec{1}]^T$ as desired. 
\end{proof}

We will now connect the null space of $X$ to the eigenspace of $K$. This will allow us to specify when the generalized eigenspace of $K$ does not connect to the generalized eigenspace of $\nonback$.

\begin{lem}\label{lemma:xv eqaul 0}
Let $\vec{u}$ be an eigenvector of $K$ for eigenvalue $\lambda$ and let $\vec{v}=\begin{bmatrix}
    \vec{u} & 0 & 0\end{bmatrix}^T$. If $X\vec{v}=\vec{0}$, then $\lambda=1$ if $G$ is not bipartite and $\lambda=\pm 1$ if $G$ is bipartite.
\end{lem}

\begin{proof}
For an eigenvector $\vec{u}$ of $K$, let $\vec v=\begin{bmatrix}\vec{u} &\vec{0} &\vec{0} \end{bmatrix}^T$. If $X\vec{v}=\vec{0}$, then $\vec{u} \in \nul\left(\begin{bmatrix}
    S & T^T
\end{bmatrix}\right).$ 

If $G$ is not bipartite, then $\nul\left(\begin{bmatrix}
    S & T^T
\end{bmatrix}\right)= \spa\left\{\begin{bmatrix}
    \vec{1} & -\vec{1}
\end{bmatrix}^T\right\}$ by Lemma \ref{lemma: null space of transition matrix}. By computation, $\begin{bmatrix}
    \vec{1} & -\vec{1}
\end{bmatrix}^T$ is a eigenvector for $K$ for the eigenvalue $\lambda = 1$. 

If $G$ is bipartite with parts $\mathcal{A}$ and $\mathcal{B}$, then \begin{align*}
    \nul\left(\begin{bmatrix}
    S & T^T
\end{bmatrix}\right)= \spa\left\{\begin{bmatrix}
    \vec{1_{|\mathcal{A}|}} & \vec{1_{|\mathcal{B}|}} &  -\vec{1_{|\mathcal{A}|}} & -\vec{1_{|\mathcal{B}|}}
\end{bmatrix}^T, \begin{bmatrix}
    \vec{1_{|\mathcal{A}|}} & -\vec{1_{|\mathcal{B}|}} &  \vec{1_{|\mathcal{A}|}} & -\vec{1_{|\mathcal{B}|}}
\end{bmatrix}^T\right\}
\end{align*} by Lemma \ref{lemma: null space of transition matrix}. By computation,  $\begin{bmatrix}
   \vec{1_{|\mathcal{A}|}} & \vec{1_{|\mathcal{B}|}} &  -\vec{1_{|\mathcal{A}|}} & -\vec{1_{|\mathcal{B}|}}
\end{bmatrix}^T$ is a eigenvector for $K$ for the eigenvalue $\lambda = 1$ and $\begin{bmatrix}
   \vec{1_{|\mathcal{A}|}} & -\vec{1_{|\mathcal{B}|}} &  \vec{1_{|\mathcal{A}|}} & -\vec{1_{|\mathcal{B}|}}
\end{bmatrix}^T$ is a eigenvector for $K$ for the eigenvalue $\lambda = -1$. Since eigenspaces are disjoint, there is no eigenvector of $K$ that is a nontrivial linear combination of these two vectors. 
\end{proof}

The next two propositions are results about the algebraic and geometric multiplicity of the eigenvalues $\pm 1$ for the non-backtracking matrix. 

\begin{prop}\cite{Torres}\label{torres1}
Let $G$ have at least two cycles. Then $\lambda=1$ is an eigenvalue for $\nonback$ with algebraic and geometric multiplicity $m - n +1$.
\end{prop}

\begin{prop}\cite{Torres}\label{torres2}
Let $G$ have at least two cycles. Then $\lambda=-1$ is an eigenvalue for $\nonback$, with algebraic and geometric multiplicity $m - n +1$ if $G$ is bipartite and algebraic and geometric multiplicity $m - n$ if $G$ is not bipartite. 
\end{prop}

The defectiveness of the eigenvalue $\lambda=0$, for the non-backtracking matrix has been explicitly found by Torres (see \cite{Torres}) for graphs with minimum degree one. Such graphs tend to be removed from the discussion about graphs with defective eigenvalues for the non-backtracking matrix because the defectiveness of $\lambda=0$ is already well understood. We include some of these graphs here for completion, but still require at least one cycle, since $M$ is only well defined when $m \geq n$. We are now able to prove our main result.

\begin{thm}\label{thm: k and b are the same}
Let $G$ be a graph. 
If $G$ has at least two cycles, then $\nonback$ and $M$ have the same Jordan form.
\end{thm}

\begin{proof}
    Let $G$ be a graph. 
    Let $\vec{u}^{(1)},\dots,\vec{u}^{(k)}$ be a Jordan chain for $K$ for the eigenvalue $\lambda$ and let $\vec{v}^{(i)}=\begin{bmatrix}
        \vec{u}^{(i)} & 0 & 0
    \end{bmatrix}^T$. Then $\vec{v}^{(1)},\dots,\vec{v}^{(k)}$ is a Jordan chain for $M$ for the eigenvalue $\lambda$. We note that since the other block of $M$ is a diagonal matrix (and thus $\vec{e}{_{2n+1}}, \ldots, \vec{e}{_{2m}}$, which are linearly independent from $\vec{v}^{(i)}$, are are eigenvectors of $M$), this is the only way to have a nontrivial Jordan chain for $M$. Also, since $M$ and $\nonback$ have the same spectrum by construction, we only need to show that the Jordan blocks are the same size.

    If $X\vec{v}^{(1)} \neq 0$, then by Proposition \ref{prop:EigenLift}, $X\vec{v}^{(1)}, \ldots, X\vec{v}^{(k)}$ is a Jordan chain of $\nonback$ for $\lambda$. 

    If $X\vec{v}^{(1)} =0$, then by Lemma \ref{lemma:xv eqaul 0}, $\lambda=1$ if $G$ is not bipartite or $\lambda=\pm 1$ if $G$ is bipartite. 

    First, let $G$ be bipartite. By Proposition \ref{torres1} and \ref{torres2}, we know the algebraic and geometric multiplicity of $\lambda=\pm 1$ for $\nonback$ is $m-n+1$. We claim that $M$ also has a full set of eigenvectors for the eigenvalues $\pm 1$. 

    If $\lambda = 1$, then eigenvectors of $M$ are $\vec{v}^{(1)}, \vec{e}{_{2n+1}}, \ldots, \vec{e_{n+m-1}}$ which form a set of $m-n+1$ linearly independent vectors. 

    If $\lambda=-1$, then eigenvectors of $M$ are $\vec{v}^{(1)}, \vec{e_{n+m}}, \ldots, \vec{e_{2m}}$ which form a set of $m-n+1$ linearly independent vectors. 

    Now, let $G$ not be biparite. By Proposition \ref{torres1} and \ref{torres2}, we know the algebraic and geometric multiplicity of $\lambda=1$ for $\nonback$ is $m-n+1$ and the algebraic and geometric multiplicity of $\lambda=-1$ for $\nonback$ is $m-n$. We claim that $M$ also has a full set of eigenvectors for the eigenvalues $\pm 1$. 
    
     If $\lambda = 1$, then eigenvectors of $M$ are $\vec{v}^{(1)}, \vec{e_{2n+1}}, \ldots, \vec{e_{n+m-1}}$ which form a set of $m-n+1$ linearly independent vectors. 

    If $\lambda=-1$, then eigenvectors of $M$ are $\vec{e_{n+m}}, \ldots, \vec{e_{2m}}$ which form a set of $m-n$linearly independent vectors. 

    Therefore, the length of every Jordan chain is the same for $M$ and $\nonback$ for every eigenvalue of $\lambda$ meaning they have the same dimension of the generalized eigenspaces for every eigenvalue of $\lambda$ and their Jordan forms are the same. 
\end{proof}

\begin{thm}
\label{thm:unicyclic}
Let $G$ not be a tree.
     The Jordan forms of $\nonback$ and $M$ differ if and only if $G$ is unicyclic and they only differ for the eigenvalue $\lambda=1$ when the cycle is odd and for the eigenvalues $\lambda=\pm 1$ when the cycle is even.
\end{thm}

\begin{proof}
    The only class of graphs that is not a tree and has less than two cycles, are the unicyclic graphs. We will now show that $\nonback$ and $M$ only differ for the eigenvalue $\lambda=1$ when the cycle is odd and for the eigenvalues $\lambda=\pm 1$ when the cycle is even.

    Torres in \cite{Torres} showed that the characteristic polynomial of a unicyclic graph of the non-backtracking matrix can be written as $p_G(x)=x^{2(n-k)}p_{C_k}(x)$ where $k$ is the size of the cycle. 
    The non-backtracking matrix of the cycle graph can be written as a block diagonal matrix with two of the diagonal blocks both being permutation matrices and the remaining ones. Therefore its eigenvalues are the $n^{th}$ roots of unity, each with multiplicity two and it has a full set of eigenvectors (since it is block diagonal).
    
    Therefore, the spectrum of the unicylic graph is $\{0^{(2(n-k))}, 2\cos(2 \pi j/k)^{(2)}\}$ for $0 \leq j \leq n-1$. Note that $\lambda=-1$ is only an eigenvalue when $n$ is even.

    In the proof of Theorem \ref{thm: k and b are the same}, we only used the fact that $G$ has at least two cycles to determine $M$ and $\nonback$ had the same Jordan form for $\lambda=\pm 1$. Therefore, these two eigenvalues are the only places where the Jordan canonical form of $\nonback$ and $M$ have the potential to differ. For the unicyclic graph, $M=K$ since it has the same number of vertices as edges. 

    For $\lambda=1$, the vector $\vec{v}=[\vec{1}_k, \vec{1}_{n-k}, -\vec{1}_k,  -\vec{1}_{n-k}]^T$ is an eigenvector for $K$. Consider $\vec{u}=[\vec{0}_k, \vec{x}_{n-k}, \vec{1}_k, \vec{1}_{n-k}-\vec{x}_{n-k}]^T$ where $\vec{x}_i=-\dist(i,j)$ where $j$ is some vertex on the cycle. We will show that $(K-I)\vec{u}=K\vec{u}-\vec{u}=\vec{v}$.
    
    First, consider a vertex $i$ that is on the cycle. Then 
    \begin{align*}
        K\vec{u}_i &= 0 + (\deg(i)-2)(-1)+(\deg(i)-1)\\
        &=1,\\
        (K\vec{u}-\vec{u})_i&=1-0=1\\
        (K\vec{u}-\vec{u})_{n+i}&=0+0-1=-1.
    \end{align*} Now let $i$ be a vertex that is not in the cycle. Note, there is only one vertex adjacent to $i$ that is distance one less to some vertex on the cycle $j$, the remaining $\deg(i)-1$ vertices are distance one more away from the cycle than $i$. Then
    \begin{align*}
        K\vec{u}_i &= (\deg(i)-1)(-1)(\dist(i,j)+1)-(\dist(i,j)-1)+(\deg(i)-1)(\dist(i,j)+1)\\
        &=1-\dist(i,j),\\
        (K\vec{u}-\vec{u})_i&=1-\dist(i,j)-(-\dist(i,j))=1\\
        (K\vec{u}-\vec{u})_{n+i}&=-(-\dist(i,j))-(\dist(i,j)+1=-1.
    \end{align*}
    Therefore, $\vec{u}$ is a generalized eigenvector for the eigenvalue $\lambda=1$. Therefore, $M$ and $\nonback$ have different Jordan blocks for $\lambda=1$. 

    For $\lambda=-1$, this is only an eigenvalue for a unicyclic graph when $k$ is even. Since the graph is bipartite, we can partition the vertices into sets $\mathcal{A},\mathcal{B}$ such that no edges are contained in $\mathcal{A}$ or contained in $\mathcal{B}$. Consider the vector $\vec{v}=[\vec{1}_{|\mathcal{A}|}, -\vec{1}_{|\mathcal{B}|},\vec{1}_{|\mathcal{A}|}, -\vec{1}_{|\mathcal{B}|}]^T$ which is an eigenvector for $\lambda$ for $K$. Consider \begin{align*}
        \vec{u}_i &=\begin{cases}
        -1 & i\in \mathcal{A} \text{ and in the cycle.}   \\
        1 & i\in \mathcal{B} \text{ and in the cycle.}  \\
        (\dist(i,j)-1) & i\in \mathcal{A} \text{ and not in the cycle.} \\
        -(\dist(i,j)-1) & i\in \mathcal{B} \text{ and not in the cycle.} 
    \end{cases}\\
    \vec{u}_{n+i} &=\begin{cases}
        0 & i\text{ is in the cycle.}   \\ 
        (\dist(i,j)) & i\in \mathcal{A} \text{ and not in the cycle.} \\
        -(\dist(i,j)) & i\in \mathcal{B} \text{ and not in the cycle.} 
    \end{cases}
    \end{align*} We will show that $(K+I)\vec{u}=K\vec{u}+\vec{u}=\vec{v}$.
    
    First, consider a vertex $i$ that is on the cycle and in $\mathcal{A}$ (if $i\in \mathcal{B}$, all the signs are reversed). Then 
    \begin{align*}
        (K\vec{u}+\vec{u})_i&=2-0-1=1\\
        (K\vec{u}+\vec{u})_{n+i}&=-(-1)+0=1.
    \end{align*} Now let $i$ be a vertex that is not in the cycle and in $\mathcal{A}$ (if $i\in \mathcal{B}$, all the signs are reversed). Note, there is only one vertex adjacent to $i$ that is distance one less to some vertex on the cycle $j$, the remaining $\deg(i)-1$ vertices are distance one more away from the cycle than $i$. Then
    \begin{align*}
        K\vec{u}_i &= (\deg(i)-1)(-1)(\dist(i,j))-(\dist(i,j)-2)+(\deg(i)-1)(\dist(i,j))\\
        &=2-\dist(i,j),\\
        (K\vec{u}+\vec{u})_i&=2-\dist(i,j)+(\dist(i,j)-1)=1\\
        (K\vec{u}-\vec{u})_{n+i}&=-(-\dist(i,j)-1)-(\dist(i,j)=1.
    \end{align*}
    
    Therefore, $\vec{u}$ is a generalized eigenvector for the eigenvalue $\lambda=-1$. Therefore, $M$ and $\nonback$ have different Jordan blocks for $\lambda=-1$ when $n$ is even. 
\end{proof}

With this result, we are able to look at a smaller matrix when investigating defective eigenvalues for the non-backtracking matrix. Further, if restricting to minimum degree at least two, $K$ and $\nonback$ are similar matrices making them spectrally interchangeable. For the class of unicyclic graphs, we completely understand the eigenspace differences between these two matrices. The only class of graphs not covered by our results is trees which are already spectrally categorized. We see this result as a major tool in future spectral study of the non-backtracking matrix. 

\section{Eigenvectors structure for $K$}\label{sec: structure for k}
In this section, we establish useful structural results about the eigenvectors and generalized eigenvectors for $K$. We see these results as being useful tools when showing that a generalized eigenvector does or does not exist. 

\subsection{Eigenvector and generalized eigenvectors} 
We begin by looking at the necessary and sufficient linear equations the values of a vector must fulfill in order to be an eigenvector of $K$. 
\begin{prop}
\label{prop:eqEigvec}
Let $\vec{v}$ be a vector in $\mathbb{C}^{2n}$ and let $\lambda\neq 0$. The vector $\vec{v}$ is an eigenvector of $K$ for the eigenvalue $\lambda$ if and only if $$\vec{v}_{n+i}=\frac{-\vec{v}_i}{\lambda}\, \text{ and } \sum_{j\sim i}\vec{v}_j=\vec{v}_i\frac{1}{\lambda}\left(\deg(i)-1+\lambda^2\right)$$ for all $1\leq i\leq n$.
\end{prop}

\begin{proof}
First assume $\vec{v}$ is an eigenvector of $K$ for the eigenvalue $\lambda$. Then $(K\vec{v})_i=\lambda \vec{v}_i$ for $1\leq i\leq 2n$. Furthermore, for $1\leq i\leq n$,
\begin{align*}
    (K\vec{v})_i&=(A\begin{bmatrix}\vec{v}_1&\cdots&\vec{v}_n\end{bmatrix}^T)_i+\left((D-I)\begin{bmatrix}\vec{v}_{n+1}&\cdots&\vec{v}_{2n}\end{bmatrix}^T\right)_i\\
    &=\sum_{j\sim i}\vec{v}_j+\vec{v}_{n+i}(\deg(i)-1)
\end{align*}
and
$$    (K\vec{v})_{n+i}=(-I\begin{bmatrix}\vec{v}_1&\cdots&\vec{v}_n\end{bmatrix}^T)_i=-\vec{v}_i.$$
Therefore
\begin{align}
    \lambda \vec{v}_i&=\sum_{j\sim i}\vec{v}_j+\vec{v}_{n+i}(\deg(i)-1) \hspace{5mm} \text{ and }\label{Eq:EigenVecUnsimp}\\
    \lambda \vec{v}_{n+i}&=-\vec{v}_i.\label{Eq:yx}
\end{align}
Solving Equation \ref{Eq:yx} for $\vec{v}_{n+i}$ and plugging it into Equation \ref{Eq:EigenVecUnsimp} and rearranging gives us 
$$\sum_{j\sim i}\vec{v}_j=\vec{v}_i\frac{1}{\lambda}\left(\deg(i)-1+\lambda^2\right),$$
as desired.

The other direction follows immediately from reversing the algebra. 
\end{proof}

We will proceed similarly to above but now looking at a generalized eigenvector. 
\begin{prop}
\label{prop:eqGenEigvec}
Let $\vec{v}$ be an eigenvector of $K$ for the eigenvalue $\lambda\not=0$ and let $\vec{u}$ be a vector in $\mathbb{C}^{2n}$. The vector $\vec{u}$ is a generalized eigenvector of $K$ such that $K\vec{u}=\lambda\vec{u}+\vec{v}$ if and only if $$\vec{u}_{n+i}=\frac{-\vec{u}_i}{\lambda}+\frac{\vec{v}_i}{\lambda^2} \, \text{ and } \, \sum_{j\sim i}\vec{u}_j=\vec{u}_i\frac{1}{\lambda}(\deg(i)-1+\lambda^2)-\vec{v}_i\frac{1}{\lambda^2}(\deg(i)-1-\lambda^2)$$
for all $1\leq i\leq n$.
\end{prop}

\begin{proof}
First, assume $\vec{u}$ be a generalized eigenvector of $K$ such that $K\vec{u}=\lambda\vec{u}+\vec{v}$. Then for $1\leq i\leq n$,
\begin{align}\label{Eq:GenFirstHalf}
(K\vec{u})_i&=\sum_{j\sim i}\vec{u}_j+\vec{u}_{n+i}(\deg(i)-1)=\lambda \vec{u}_i+\vec{v}_i
\end{align}
and 
\begin{align}\label{Eq:GenSecondHalf}
(K\vec{u})_{i+n}&=-\vec{u}_i=\lambda \vec{u}_{n+i}+\vec{v}_{n+i}.
\end{align}
Recalling that $\vec{v}_{n+i}=\frac{-\vec{v}_i}{\lambda}$ and plugging this in to Equation \ref{Eq:GenSecondHalf} and rearranging yields
\begin{align*}
    -\vec{u}_i&=\lambda \vec{u}_{n+i}-\frac{\vec{v}_i}{\lambda}.
\end{align*}
Plugging this in to Equation \ref{Eq:GenFirstHalf}, we get
\begin{align*}
    \sum_{j\sim i}\vec{u}_j+\left(\frac{-\vec{u}_i}{\lambda}+\frac{\vec{v}_i}{\lambda^2}\right)(\deg(i)-1)&=\lambda \vec{u}_i+\vec{v}_i,\\
    \text{ so  } \, \, \sum_{j\sim i}\vec{u}_j&=\vec{u}_i\frac{1}{\lambda}(\deg(i)-1+\lambda^2)-\vec{v}_i\frac{1}{\lambda^2}(\deg(i)-1-\lambda^2),
\end{align*} as desired.

The other direction follows immediately from reversing the algebra. \end{proof}

\begin{cor}
\label{cor:veciff}
    Let $\vec{v}$ and $\vec{u}$ be vectors in $\mathbb{C}^{2n}$ and let $\lambda\neq 0$. The vectors $\vec{v}$ and $\vec{u}$ satisfy the claimed equations from Propositions \ref{prop:eqEigvec} and \ref{prop:eqGenEigvec} if and only if $(\lambda,\vec{v},\vec{u})$ forms an generalized eigenpair for the matrix $K$.
\end{cor}

Again, we have shown the necessary and sufficient conditions two vectors must have in order to be within a Jordan chain of each other. 

\subsection{Twins}
Now we will investigate the structure of the eigenvectors and generalized eigenvectors when we have a special graph structure. Two vertices in a graph $G$, $x$ and $y$ are called \emph{(adjacent) twins} if their neighborhoods in $G\backslash \{x,y\}$ are equivalent (and they are adjacent). Twins in graphs relate nicely to matrices because the two columns corresponding to the twins have identical entries (except for the $x,y$ principle submatrix). 

We will start by showing that if a graph has twins, we have a known eigenvector entries and a known eigenvalue.

\begin{prop}\label{Prop:NonTwinEigen}
Let $G$ be a graph on $n$ vertices containing a pair of (adjacent) twin vertices $x$ and $y$ such that $\deg(x)=\deg(y)=d$, and let $\vec{v}$ be a vector in $\mathbb{C}^{2n}$. Then the vector $\vec{v}$ with entries $\vec{v}_x=1$, $\vec{v}_{n+x}=\frac{-1}{\lambda}$, $\vec{v}_y=-1$, $\vec{v}_{n+y}=\frac{1}{\lambda}$, and $\vec{v}_k=\vec{v}_{n+k}=0$ for all $1\leq k\leq n$ such that $k\not\in\{x,y\}$ is an eigenvector of $K(G)$ for the eigenvalue $\lambda=\pm\sqrt{1-d}$ if $x,y$ are non-adjacent (if $x,y$ are adjacent, then $\lambda=\frac{-1\pm \sqrt{5-4d}}{2}$).
Furthermore, if $\vec{v}$ is an eigenvector of $K(G)$ for an eigenvalue $\lambda' \not\in \{\lambda,0\}$, then $\vec{v}_x=\vec{v}_y$.
\end{prop}

\begin{proof}
First, let $\vec{v}_x=1$, $\vec{v}_{n+x}=\frac{-1}{\lambda}$, $\vec{v}_y=-1$, $\vec{v}_{n+y}=\frac{1}{\lambda}$, and $\vec{v}_k=\vec{v}_{n+k}=0$ for all $1\leq k\leq n$ such that $k\not\in\{x,y\}$. We will show that such a vector $\vec{v}$ is an eigenvector for $K(G)$ for the eigenvalue $\lambda$ by showing the equations in Proposition \ref{prop:eqEigvec} hold. It is obvious that $\vec{v}_{n+i}=\frac{-\vec{v}_i}{\lambda}$ for all $1\leq i\leq n$, so we turn our attention to the second equation. Since $\vec{v}_x$ and $\vec{v}_y$ correspond to non-adjacent twins, for each vertex $k$ in their neighborhood, $$\sum_{j\sim k}\vec{v}_j=1-1=0=\vec{v}_k.$$ For vertices not in their neighborhood, there are no non-zero terms in the equations. Finally, for the case $x,y$ are non-adjacent consider the equations for $\vec{v}_x$ and $\vec{v}_y$, 
$$\sum_{j\sim x}\vec{v}_j=0=1\frac{1}{\lambda}\left(d-1+\lambda^2\right) \, \text{ and } \, \sum_{k\sim y}\vec{v}_j=0=-1\frac{1}{\lambda}\left(d-1+\lambda^2\right).$$
Both of these equations hold when $\lambda=\pm\sqrt{1-d}$, so the described vector is an eigenvector of $K$ by Proposition \ref{prop:eqEigvec}.

For the case $x,y$ are adjacent consider the equations for $\vec{v}_x$ and $\vec{v}_y$,
$$\sum_{j\sim x}\vec{v}_j=-1=1\frac{1}{\lambda}\left(d-1+\lambda^2\right) \, \text{ and } \, \sum_{k\sim y}\vec{v}_k=1=-1\frac{1}{\lambda}\left(d-1+\lambda^2\right).$$
Both of these equations hold when $\lambda=-\frac{1}{2}\pm\frac{\sqrt{5-4d}}{2}$, so the described vector is an eigenvector of $K$ by Proposition \ref{prop:eqEigvec}.

Next, let $\vec{v}$ be an eigenvector of $K(G)$ for the eigenvalue $\lambda' \neq \lambda$. Applying Proposition \ref{prop:eqEigvec}, the equations corresponding to the vertices $x$ and $y$ are
$$\sum_{j\sim x}\vec{v}_j=\vec{v}_x\frac{1}{\lambda'}\left(d-1+(\lambda')^2\right) \, \text{ and } \, \sum_{k\sim y}\vec{v}_k=\vec{v}_y\frac{1}{\lambda'}\left(d-1+(\lambda')^2\right).$$
If the vertices are non-adjacent twins, $$\sum_{j\sim x}\vec{v}_j=\sum_{k\sim y}\vec{v}_k$$ and if the vertices are adjacent twins \begin{align*}
    \sum_{j\sim x}\vec{v}_j-\sum_{k\sim y}\vec{v}_k = \vec{v}_y-\vec{v}_x.
\end{align*}
Subtracting the two equations for the non-adjacent case we get 
\begin{align}\label{Eq:NonAdjTwinEig}
\vec{v}_x\frac{1}{\lambda'}\left(d-1+(\lambda')^2\right)&=\vec{v}_y\frac{1}{\lambda'}\left(d-1+(\lambda')^2\right).
\end{align} Since $\lambda'\not=\pm\sqrt{1-d}$, we have $d-1+(\lambda')^2\not=0$. Therefore, $\vec{v}_x=\vec{v}_y$.

Subtracting the two equations for the adjacent case we get
\begin{align}\label{Eq:AdjTwinEig}
\vec{v}_x\frac{1}{\lambda'}\left(d-1+(\lambda')^2+\lambda'\right)&=\vec{v}_y\frac{1}{\lambda'}\left(d-1+(\lambda')^2+\lambda'\right).
\end{align} 
Since $\lambda'\not=-\frac{1}{2}\pm\frac{\sqrt{5-4d}}{2}$, we have $d-1+(\lambda')^2+\lambda'\not=0$. Therefore, $\vec{v}_x=\vec{v}_y$, as desired.
\end{proof}

In the case that $K(G)$ has a defective eigenvalue $\lambda$ with eigenvector $\vec{v}$ and generalized eigenvector $\vec{u}$, we can further determine the affect of the non-adjacent twin structure on the entries of $\vec{v}$ and $\vec{u}$.

\begin{prop}\label{Prop:NonAdjTwinGenEigen}
Let $G$ be a graph containing a pair of (adjacent) twin vertices $x$ and $y$ such that $\deg(x)=\deg(y)=d$ and let $(\lambda,\vec{v},\vec{u})$ be an generalized eigenpair for the matrix $K(G)$ such that $\lambda\not=0$. If $\lambda=\pm\sqrt{1-d}$ and $x,y$ are non-ajdacent, then $\vec{v}_x=\vec{v}_y$ and if $\lambda' \not\in \{\lambda,0\}$, then $\vec{u}_x=\vec{u}_y$. If $\lambda=-\frac{1}{2}\pm\frac{\sqrt{5-4d}}{2}$ and $x,y$ are adjacent, then $\vec{v}_x=\vec{v}_y$ and if $\lambda' \not\in \{\lambda,0\}$, then $\vec{u}_x=\vec{u}_y$.
\end{prop}

\begin{proof}
Let $(\lambda,\vec{v},\vec{u})$ be an generalized eigenpair for the matrix $K(G)$. Applying Proposition \ref{prop:eqGenEigvec}, the equations corresponding to the vertices $x$ and $y$ are

$$\sum_{j\sim x}\vec{u}_j=\vec{u}_x\frac{1}{\lambda}(d-1+\lambda^2)-\vec{v}_x\frac{1}{\lambda^2}(d-1-\lambda^2) \text{ and } \sum_{k\sim y}\vec{u}_k=\vec{u}_y\frac{1}{\lambda}(d-1+\lambda^2)-\vec{v}_y\frac{1}{\lambda^2}(d-1-\lambda^2)$$

First, let $x,y$ be non-adjacent twins. Therefore $$\sum_{j\sim x}\vec{u}_j=\sum_{k\sim y}\vec{u}_k$$

and subtracting the two equations we get

\begin{align}\label{eq:genEigen}
   (\vec{u}_x-\vec{u}_y)\frac{(d-1+\lambda^2)}{\lambda}=(\vec{v}_x-\vec{v}_y)\frac{(d-1-\lambda^2)}{\lambda^2}.
\end{align}

If $\lambda=\pm\sqrt{1-d}$, then $\frac{(d-1+\lambda^2)}{\lambda}=0$. Note that in this case, $\frac{(d-1-\lambda^2)}{\lambda^2}=\frac{d-1-1+d}{1-d}=-2\not=0$, so it must be that $\vec{v}_x=\vec{v}_y$. If $\lambda\not=\pm\sqrt{1-d}$, $\frac{(d-1+\lambda^2)}{\lambda}\not=0$ and by Proposition \ref{Prop:NonTwinEigen}, $\vec{v}_x=\vec{v}_y$. Therefore, in this case, it must be that $\vec{u}_x=\vec{u}_y$.

Finally, let $x,y$ be adjacent twins. Therefore, $$\sum_{j\sim x}\vec{u}_j-\sum_{k\sim y}\vec{u}_k=\vec{u}_y-\vec{u}_x$$

and subtracting the two equations we get
\begin{align*}
   (\vec{u}_x-\vec{u}_y)\frac{(d-1+\lambda^2+\lambda)}{\lambda}=(\vec{v}_y-\vec{v}_x)\frac{(d-1-\lambda^2)}{\lambda^2}.
\end{align*}

If $\lambda=-\frac{1}{2}\pm\frac{\sqrt{5-4d}}{2}$, then $\frac{(d-1+\lambda^2+\lambda)}{\lambda}=0$. Note that in this case, 
$$\frac{(d-1-\lambda^2)}{\lambda^2}=\frac{1}{\lambda}\left(\frac{(d-1+\lambda^2+\lambda)}{\lambda} \right)-2-\frac{1}{\lambda}=-2-\frac{1}{\lambda}\not=0,$$ so it must be that $\vec{v}_x=\vec{v}_y$. If $\lambda\not=-\frac{1}{2}\pm\frac{\sqrt{5-4d}}{2}$, then $\frac{(d-1+\lambda^2+\lambda)}{\lambda}\not=0$ and by Proposition \ref{Prop:NonTwinEigen}, $\vec{v}_x=\vec{v}_y$. Therefore, in this case, it must be that $\vec{u}_x=\vec{u}_y$.
\end{proof}

These results helps us better understand the composition of the eigenvectors and generalized eigenvectors of $K$. We see these results being useful in showing that some vector is or is not an eigenvector or generalized eigenvector. 

\section{Defective Graph Families}\label{sec: fams}

In this section, we give some examples of infinite graph families which have a nontrivial Jordan block for some eigenvalue. We also provide computational results regarding the number of such graphs on ten or fewer vertices.

To define infinite families with nontrivial Jordan blocks, we will define the idea of \emph{graph gluing}. Consider a graph $G$ with vertex $x$ and a graph $H$ with a vertex $y$. We can construct a new graph $G\glue{x}{y} H$ by identifying $x$ and $y$ together. Specifically, \[V(G \glue{x}{y} H) = V(G) \cup V(H)\setminus \{y\}\] and \[E(G\glue{x}{y} H)=E(G)\cup E(H) \cup \{vx : vy \in E(H) \} \setminus \{vy \in E(H)\}\}.\]

If we would like to glue two graphs together at several vertices, we can define two lists of vertices $X=[x_1,x_2,\ldots x_k]$ and $Y=[y_1,y_2,\ldots y_k]$ where $x_i \in V(G)$ and $y_i \in V(H)$ for all $i$. Then $G\glue{X}{Y} H$ is the graph formed by gluing $x_i$ to $y_i$ for all $i$, ignoring any multiedges that might arise. See an example of this gluing in Figure~\ref{fig:gluing}.

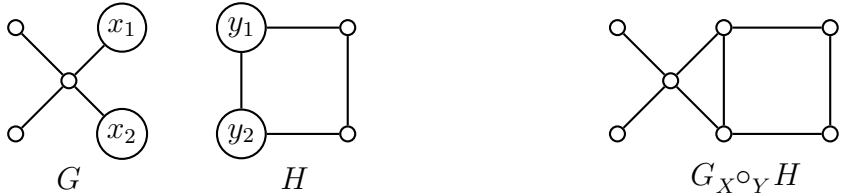
\begin{figure}[h!]
    \centering
\begin{tikzpicture}
\tikzstyle{knode}=[circle,draw=black,thick,inner sep=2pt]
\node (c1) at (-5,0) {};
\node (labelG) at ($(c1)+(0,-1.3)$) {$G$};
\node (c2) at (-2,0) {};
\node (labelH) at ($(c2)+(0,-1.3)$) {$H$};
\node (0a) at ($(c1)+(0:0cm)$) [knode] {};
\node (1a) at ($(c1)+(45:1cm)$) [knode] {\small $x_1$};
\node (2a) at ($(c1)+(135:1cm)$) [knode] {};
\node (3a) at ($(c1)+(-135:1cm)$) [knode] {};
\node (4a) at ($(c1)+(-45:1cm)$) [knode] {$x_2$};
\node (1b) at ($(c2)+(45:1cm)$) [knode] {};
\node (2b) at ($(c2)+(135:1cm)$) [knode] {$y_1$};
\node (3b) at ($(c2)+(-135:1cm)$) [knode] {$y_2$};
\node (4b) at ($(c2)+(-45:1cm)$) [knode] {};

\foreach \from/\to in {0a/1a,0a/2a,0a/3a,0a/4a, 1b/2b,2b/3b,4b/1b,3b/4b}
\draw[thick] (\from)--(\to);

\node (c3) at (3,0) {};
\node (labelGH) at ($(c3)+(1,-1.3)$) {$G \glue{X}{Y}H$};
\node (0a) at ($(c3)+(0:0cm)$) [knode] {};
\node (1a) at ($(c3)+(45:1cm)$) [knode] {};
\node (2a) at ($(c3)+(135:1cm)$) [knode] {};
\node (3a) at ($(c3)+(-135:1cm)$) [knode] {};
\node (4a) at ($(c3)+(-45:1cm)$) [knode] {};
\node (1b) at ($(1a)+(0:1.414cm)$) [knode] {};
\node (4b) at ($(4a)+(0:1.414cm)$) [knode] {};

\foreach \from/\to in {0a/1a,0a/2a,0a/3a,0a/4a, 1b/1a,1a/4a,4b/1b,4a/4b}
\draw[thick] (\from)--(\to);
\end{tikzpicture}
    \caption{Two graphs $G$ and $H$ with indicated vertex lists $X$ and $Y$, respectively. Their glued result is on the right.}
    \label{fig:gluing}
\end{figure}

We now show that gluing any graph to a defective graph, under certain conditions, will result in a defective graph.

\begin{thm}\label{thm: gluing}
    Let $G$ be a graph with vertex $x$. Let $\vec{u}^{(1)}, \ldots, \vec{u}^{(k)}$ be a Jordan chain of $K(G)$ for the eigenvalue $\lambda\not=0$. For a graph $H$ with vertex $y$, let $G \glue{x}{y} H$ be the graph where $G$ and $H$ have been glued at $x$ and $y$, respectively.
    If $\vec{u}^{(i)}_x=0$ for all $i$, then $G \glue{x}{y} H$ has a Jordan chain of length $k$ for the eigenvalue $\lambda$ for the matrix $K(G\glue{x}{y} H)$.  
\end{thm}

\begin{proof}
    Let $G$ be a graph with vertex $x$ and $\vec{u}^{(1)}, \ldots, \vec{u}^{(k)}$ be a Jordan chain of $K(G)$ for the eigenvalue $\lambda\not=0$ with $\vec{u}^{(i)}_x=0$ for all $1\leq i\leq k$. Since $\vec{u}^{(1)}$ is an eigenvector and $\lambda\not=0$, $\vec{u}^{(1)}_{n+x}=0$ by Proposition \ref{prop:eqEigvec}. Furthermore, since $-\vec{u}_x^{(i+1)}=\lambda\vec{u}_{n+x}^{(i+1)}+\vec{u}_{n+x}^{(i)}$ for all $1 \leq i< k$, it is clear by induction that $\vec{u}_{n+x}^{(i)}=0$ for all $1\leq i\leq k$.

    Consider $G \glue{x}{y} H$ for a graph $H$ with vertex $y$ and let $n'$ be the number of vertices in $G \glue{x}{y} H$. We construct $\vec{v}^{(i)}$ for $G \glue{x}{y} H$ such that $\vec{v}^{(i)}_j=\vec{u}^{(i)}_j$ and $\vec{v}^{(i)}_{n'+j}=\vec{u}^{(i)}_{n+j}$ for $j \in V(G)$ and $\vec{v}^{(i)}_j=\vec{v}^{(i)}_{n'+j}=0$ otherwise. That is, the new vectors will match the old vectors on $G$ and be zero on $H$. 

    For $G \glue{x}{y} H$, consider the $j$th entry of $K(G \glue{x}{y} H)\vec{v}^{(i)}$ which is $
        (\deg(j)-1)\vec{v}^{(i)}_{ n'+j}+\sum_{k \sim j} \vec{v}^{(i)}_k$, and the $(n'+j)$th entry is
    $ -\vec{v}^{(i)}_j$. 

To prove that $\{\vec{v}^{(i)}\}$ is a Jordan chain, we consider the vertices in two groups: vertices of $G$ that are not $x$, and vertices of $H$ including $x$. First, we will first show that for $j \in V(G)$ and $j \neq x$, 
\begin{equation}
\label{eq:liftG1}
    (\deg(j)-1)\vec{v}^{(i)}_{ n'+j}+\sum_{k \sim j} \vec{v}^{(i)}_k= \lambda \vec{v}^{(i)}_j +\vec{v}^{(i-1)}_j
\end{equation}
and
\begin{equation}
\label{eq:liftG2}
-\vec{v}^{(i)}_j=\lambda \vec{v}^{(i)}_{n'+j}+ \vec{v}^{(i-1)}_{n'+j}.    
\end{equation}

Then, we will also show for $j \in V(H)$ or $j=x$,
\begin{equation}
\label{eq:liftH1}
(\deg(j)-1)\vec{v}^{(i)}_{ n'+j}+\sum_{k \sim j} \vec{v}^{(i)}_k= \lambda \vec{v}^{(i)}_j +\vec{v}^{(i-1)}_j=0
\end{equation}
and
\begin{equation}
\label{eq:liftH2}
    -\vec{v}^{(i)}_j=\lambda \vec{v}^{(i)}_{n+j}+ \vec{v}^{(i-1)}_{n+j}=0.
\end{equation}

To prove equations \eqref{eq:liftG1} and \eqref{eq:liftG2}, let $j \in V(G)$ and $j \neq x$. Then, because the vectors $\vec{u}^{(i)}$ form a Jordan chain ,  
\begin{align*}
    (\deg(j)-1)\vec{v}^{(i)}_{ n'+j}+\sum_{k \sim j} \vec{v}^{(i)}_k&=(\deg(j)-1)\vec{u}^{(i)}_{ n+j}+\sum_{k \sim j} \vec{u}^{(i)}_k\\
    &=\lambda \vec{u}^{(i)}_j +\vec{u}^{(i-1)}_j =\lambda \vec{v}^{(i)}_j +\vec{v}^{(i-1)}_j,
\end{align*} and 
    \[-\vec{v}^{(i)}_j=-\vec{u}^{(i)}_j=\lambda \vec{u}^{(i)}_{n+j}+ \vec{u}^{(i-1)}_{n+j}=\lambda \vec{v}^{(i)}_{n'+j}+ \vec{v}^{(i-1)}_{n'+j}.\]

Now, to prove equations \eqref{eq:liftH1} and \eqref{eq:liftH2},  let $j \in V(H)$ or $j = x$. Then 
    \begin{align*}
    (\deg(j)-1)\vec{v}^{(i)}_{ n'+j}+\sum_{k \sim j} \vec{v}^{(i)}_k&=(\deg(j)-1)(0)+\sum_{k \sim j, \: k \in V(G)} \vec{v}^{(i)}_k+\sum_{k \sim j, \: k \in V(H)} \vec{v}^{(i)}_k\\
    &=\sum_{k \sim j, \: k \in V(G)} \vec{u}^{(i)}_k+\sum_{k \sim j, \: k \in V(H)} 0\\
    &=0\\
    &=\lambda \vec{v}^{(i)}_j +\vec{v}^{(i-1)}_j,
\end{align*} 
since $\sum_{k \sim x, k \in V(G)} \vec{u}^{(i)}_k=0$ by our assumption on the vectors $\{\vec{u}^{(i)}\}$. Further,  
\[-\vec{v}^{(i)}_j=0 =\lambda \vec{v}^{(i)}_{n'+j}+ \vec{v}^{(i-1)}_{n'+j}.\]
Therefore, we have constructed a Jordan chain of length $k$ for $\lambda$ for the matrix $K(G \glue{x}{y} H)$.
\end{proof}

\begin{cor}
\label{cor:gluing}
    Let $G$ be a graph with $X \subseteq V(G)$ such that for all $x \in X$ and all $i$, $\vec{u}^{(i)}_x=0$ where $\vec{u}^{(1)}, \ldots \vec{u}^{(k)}$ form a Jordan chain of $K(G)$ for the eigenvalue $\lambda\not=0$. Let $H$ be another graph with a subset of vertices $Y$ where $|X|=|Y|$. 
    
    Construct $G \glue{X}{Y} H$ by gluing the vertices $X$ to the vertices $Y$ as described. Then $G \glue{X}{Y} H$ has a Jordan chain of length $k$ for the eigenvalue $\lambda$ for the matrix $K(G \glue{X}{Y} H)$.
\end{cor}

    This is verified in the same manner as Theorem \ref{thm: gluing}. 

\subsection{Bipartite Base Family}
 Let $\mathcal{F}_{n}$ be the family of graphs built through the following process:
\begin{enumerate}
    \item Start with $G$, a $4$-regular bipartite graph on $2n$ vertices, and let $\mathcal{A}=\{a_1,a_2,\dots,a_n\}$ and $\mathcal{B}=\{b_1,b_2,\dots,b_n\}$ be the partite sets.
    \item Let $H$ be a graph on $\ell\geq 1$ vertices. Take the disjoint union $G\cup H$.
    \item For all $1\leq i\leq n$, add an edge from $a_i\in \mathcal{A}$ to any single vertex in $H$ and from $b_i\in \mathcal{B}$ to any single vertex in $H$, such that each vertex in $H$ is adjacent to the same number of vertices in $\mathcal{A}$ and $\mathcal{B}$ and such that the resulting graph is connected.
\end{enumerate}

An immediate example of this construction is $K_{4,4}+K_1$, but another example is shown in Figure~\ref{fig:bipartite}.

\begin{figure}[h!]
    \centering
    \begin{tikzpicture}
        \tikzstyle{knode}=[circle,draw=black,thick,inner sep=2pt]
        \node (a0) at (-2,-1) [knode] {$a_1$};
        \node (a1) at (-1,-1) [knode] {$a_2$};
        \node (a2) at (0,-1) [knode] {$a_3$};
        \node (a3) at (1,-1) [knode] {$a_4$};
        \node (a4) at (2,-1) [knode] {$a_5$};
        \node (b0) at (-2,1) [knode] {$b_1$};
        \node (b1) at (-1,1) [knode] {$b_2$};
        \node (b2) at (0,1) [knode] {$b_3$};
        \node (b3) at (1,1) [knode] {$b_4$};
        \node (b4) at (2,1) [knode] {$b_5$};
        \foreach \from/\to in {0/1,0/2,0/3,0/4,1/2,1/3,1/4,2/3,2/4,3/4}{
        \draw[thick] (a\from)--(b\to);
        \draw[thick] (b\from)--(a\to);}
        \node (h1) at (-3,0) [knode] {$h_1$};
        \node (h2) at (3,0) [knode] {$h_2$};
        \draw[thick] (h1)--(h2);
        \draw[thick] (a0)--(h1)--(b0);
        \draw[thick] (a4)--(h2)--(b4);

        \draw[thick] (h1) to [out=120,in=120] (b3);
        \draw[thick] (h1) to [out=90,in=110] (b2);
        \draw[thick] (h1) to [out=-90,in=-100] (a1);
        \draw[thick] (h1) to [out=-130,in=-90] (a2);

        \draw[thick] (h2) to [out=90, in=90] (b1);
        \draw[thick] (h2) to [out=-90,in=-90] (a3);
        
    \end{tikzpicture}
    \caption{An graph in $\mathcal{F}_5$ constructed by (1) starting with $K_{5,5}$ minus one matching (specifically $a_ib_i$), (2) allowing $H=K_2$ and (3) adding an edges between $h_1$ and $\{a_1,a_2,a_3,b_1,b_3,b_4\}$ and $h_2$ and $\{a_4,a_5,b_2,b_5\}$.}
    \label{fig:bipartite}
\end{figure}
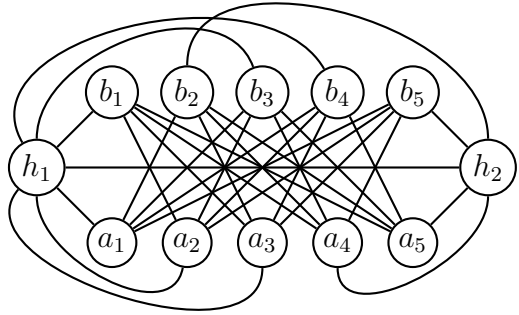

\begin{prop}
Any graph in $\mathcal{F}_n$ has a Jordan chain $(\lambda, \vec{v}, \vec{u})$ for $K$ for $\lambda=-2$, 
    \[\vec{v}_j=\left\{\hspace{-5pt} \begin{array}{cl} 1 & j \in \mathcal{A} \\ -1 & j \in \mathcal{B}   \\ 0 & \text{else} \end{array} \right. \hspace{.1 in} \vec{v}_{n'+j}=\left\{ \hspace{-5pt}\begin{array}{cl} -1/2 & j \in \mathcal{A} \\ 1/2 & j \in \mathcal{B}   \\ 0 & \text{else} \end{array} \right. \hspace{.1 in} \vec{u}_j=\left\{\hspace{-5pt} \begin{array}{cl} -1/2 & j \in \mathcal{A} \\ 1/2 & j \in \mathcal{B}   \\ 0 & \text{else} \end{array} \right. \hspace{.1 in} \vec{u}_{n'+j}=\left\{ \hspace{-5pt}\begin{array}{cl} 0 & j \in \mathcal{A} \\ 0 & j \in \mathcal{B}   \\ 0 & \text{else} \end{array} \right.\] 
    where $n'$ is the total number of vertices in the graph.
\end{prop}

The proof of this proposition follows from Corollary~\ref{cor:veciff}.
\subsection{Crustacean Family}

Let $G$ be one of the graphs shown in Figure \ref{fig: crustacean}, and define $X$ for either graph to be the vertices labeled $\ast$. Both of these graphs can serve as the graph $G$ in Corollary~\ref{cor:gluing}, as we see below. It is interesting to note that graph (a) can be constructed from graph (b) by gluing the top $\ast$ vertex (adjacent to $0$ and $1$) to one of the other $\ast$ vertices.

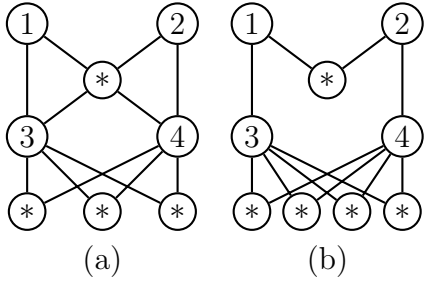
\begin{figure}[h!]
    \centering
    \begin{tabular}{cc}
       \begin{tikzpicture}
           \node[vertex] (v1) at (0,0) {3};
           \node[vertex] (v2) at (2,0) {4};
           \node[vertex] (v3) at (0,1.5) {1};
           \node[vertex] (v4) at (2,1.5) {2};
           \node[vertex] (v5) at (1,0.75) {$\ast$};
           \node[vertex] (v6) at (0,-1) {$\ast$};
           \node[vertex] (v7) at (1,-1) {$\ast$};
           \node[vertex] (v8) at (2,-1) {$\ast$};
           \draw[thick] (v1)--(v6)--(v2)--(v7)--(v1)--(v8)--(v2)--(v5)--(v1)--(v3)--(v5)--(v4)--(v2);
       \end{tikzpicture}  &  \begin{tikzpicture}
           \node[vertex] (v1) at (0,0) {3};
           \node[vertex] (v2) at (2,0) {4};
           \node[vertex] (v3) at (0,1.5) {1};
           \node[vertex] (v4) at (2,1.5) {2};
           \node[vertex] (v5) at (1,0.75) {$\ast$};
           \node[vertex] (v6) at (0,-1) {$\ast$};
           \node[vertex] (v7) at (.66,-1) {$\ast$};
           \node[vertex] (v8) at (1.33,-1) {$\ast$};
           \node[vertex] (v9) at (2,-1) {$\ast$};
           \draw[thick] (v1)--(v6)--(v2)--(v7)--(v1)--(v8)--(v2)--(v9)--(v1)--(v3)--(v5)--(v4)--(v2);
       \end{tikzpicture}\\
        (a) & (b) 
    \end{tabular}
    
    \caption{Two different base graphs with gluing set $\{\ast\}$.}
    \label{fig: crustacean}
\end{figure}

\begin{prop}
    Let $G$ be either of the graphs in Figure~\ref{fig: crustacean}. Then $(\lambda, \vec{v},\vec{u})$ is a Jordan chain for $K(G)$, where $\lambda=\sqrt{2}i$,
    \[\vec{v}_j=\left\{\hspace{-5pt} \begin{array}{cl} 1 & j=1 \\ -1 & j=2   \\ -\lambda/2 & j=3 \\ \lambda/2 & j=4 \\ 0 & \text{else}\end{array} \right. \hspace{.1 in} \vec{v}_{n'+j}=\left\{ \hspace{-5pt}\begin{array}{cl} \lambda/2 & j=1 \\ -\lambda/2 & j=2   \\ 1/2 & j=3 \\ -1/2 & j=4 \\ 0 & \text{else} \end{array} \right. \hspace{.1 in} \vec{u}_j=\left\{\hspace{-5pt}\begin{array}{cl} \lambda & j=1 \\ -\lambda & j=2   \\ -1/2 & j=3 \\ 1/2 & j=4 \\ 0 & \text{else} \end{array} \right. \hspace{.1 in} \vec{u}_{n'+j}=\left\{\hspace{-5pt} \begin{array}{cl} -3/2 & j=1 \\ 3/2 & j=2   \\ 0 & j=3 \\ 0 & j=4 \\ 0 &  \text{else}. \end{array} \right.\]
\end{prop}

The proof of this claim uses Corollary~\ref{cor:veciff}, Corollary~\ref{cor:gluing}, and straightforward algebra.

\begin{cor}
    Let $G$ be either of the graphs in Figure~\ref{fig: crustacean}. Then $(\lambda, \vec{v},\vec{u})$ is a Jordan chain for $K(G)$, where $\lambda=-\sqrt{2}i$ and $\vec{v}$ and $\vec{u}$ are defined as above.
\end{cor}
This follows naturally for complex eigenvalues of real-valued matrices. 

\subsection{Restricted Diamonds Family}
Our last family explains two of the three defective graphs on seven vertices (the third being the cycle $C_7$) and again makes use of Corollary~\ref{cor:gluing}.

\begin{figure}[h!]
    \centering
        \begin{tikzpicture}
           \draw[thick] (-1,-1) arc (180:90:2);
           \node[vertex] (v1) at (0,0) {$4$};
           \node[vertex] (v2) at (0,-2) {$5$};
           \node[vertex] (v3) at (-1,-1) {$\ast$}; 
           \node[vertex] (v4) at (1,-1) {$\ast$}; 
           \node[vertex] (v5) at (0,2) {$1$};
           \node[vertex] (v6) at (-1,1) {$2$}; 
           \node[vertex] (v7) at (1,1) {$3$};  
           \draw[thick] (v2)--(v1)--(v3)--(v2)--(v4)--(v1)--(v6)--(v5)--(v7)--(v6);
           \draw[thick] (v7)--(v1);
           \draw[thick] (v7)--(v4);
           \end{tikzpicture}
    \caption{Restricted diamonds base graph with gluing set $\{\ast\}$}
    \label{fig:RestDiamond}
\end{figure}
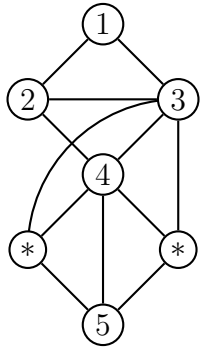

\begin{prop}
\label{prop:RestDiamond} Let $G$ be the graph in Figure~\ref{fig:RestDiamond}. Then there exists a Jordan chain $(\lambda,\vec{v},\vec{u})$ for $\lambda=\frac{-1-\sqrt{7}i}{2}$
\begin{align*}
    \vec{v}_j &=\left\{\hspace{-5pt} \begin{array}{cl} \lambda & j=1 \\ 2/\lambda & j=2   \\ 0 & j=3 \\ 1 & j=4 \\ -1 & j=5 \\0 & \text{else}\end{array} \right. \hspace{0.6 in} \vec{v}_{n'+j}=\left\{ \hspace{-5pt}\begin{array}{cl} -1 & j=1 \\ -2/\lambda^2 & j=2   \\ 0 & j=3 \\ -1/\lambda & j=4 \\ 1/\lambda & j=5 \\0 & \text{else}\end{array} \right.\\
    \vec{u}_j&=\left\{ \hspace{-5pt} \begin{array}{cl} 1 & j=1 \\ (5\lambda-3)/2 & j=2   \\ (3-\lambda)/2 & j=3 \\ -7(\lambda+1)/2 & j=4 \\ 4\lambda+2 & j=5 \\0 & \text{else}\end{array} \right.  \vec{u}_{n'+j}=\left\{\hspace{-5pt} \begin{array}{cl} 0 & j=1 \\ -(\lambda+5)/2 & j=2   \\ (3\lambda+5)/4 & j=3 \\ 3(1-\lambda)/2 & j=4 \\ (3\lambda-11)/4 & j=5 \\0 & \text{else}.\end{array} \right.
\end{align*}

\end{prop}

The proof of this claim uses Corollary~\ref{cor:veciff}, Corollary~\ref{cor:gluing}, the fact that $\lambda$ satisfies $x^2+x+2=0$, and straightforward algebra.

\begin{cor}
    Let $G$ be the graph in Figure~\ref{fig:RestDiamond}. Then there exists a Jordan chain $(\lambda,\vec{v},\vec{u})$ for $\lambda=\frac{-1+\sqrt{7}i}{2}$ and $\vec{v}$ and $\vec{u}$ are  defined as above.
\end{cor}
This follows naturally for complex eigenvalues of real-valued matrices.

\subsection{Computational Results}
We now provide some computational results for defective graphs for $K$ with minimum degree at least two. Table~\ref{tab:total} shows that the existence of defective eigenvalues for these graphs is relatively rare. Recall that the cycle is the only graph which is degenerate for $K$ and not $\nonback$, so the count of defective graphs for $\nonback$ is one less than shown.

\begin{table}[H]
    \centering
\begin{tabular}{|c||c|c|c|c|}
\hline
Number of vertices & 7 & 8 & 9 & 10 \\ \hline \hline
Defective & 3 & 39 & 484 & 7280  \\ \hline
Total & 507 & 7442 & 197772 & 9808209 \\ \hline
\end{tabular}
    \caption{The number of connected graphs with minimum degree two on $n=7,8,9,10$ vertices and the number those graphs which are defective for $K$.}
    \label{tab:total}
\end{table}

Table~\ref{tab:comp data} shows all values that manifest as defective eigenvalues for a graph on nine or fewer vertices, along with the number of defective graphs. We also give partial data for graphs on ten vertices for those same eigenvalues. This table omits 116 different graphs on ten vertices with other defective eigenvalues.

\begin{table}[ht]
    \centering
    \begin{tabular}{|c||r|r|r|r|}
    \hline
    \diagbox{$\downarrow \lambda$}{$n\rightarrow$} & $7$ & $8$ & $9$ & $10$ \\
    \hline
    \hline
        $1$ & 1 & 1 & 1&1 \\
    \hline
    $-1$ & & 1 & &1 \\
    \hline
    $(-1 \pm \sqrt{7}i)/2$ & $2$ & $16$ & $156$& 1918\\
    \hline
    $\pm \sqrt{2}i$ &  & $22$ & $324$& 5063 \\
    \hline
    $\pm \sqrt{3}i$ &  &  & $3$ & 183\\
    \hline
    $-2$ &  &  & $1$ & 5\\
     \hline
    \end{tabular}
    \caption{The number of graphs on $n$ vertices with $\lambda$ as a defective eigenvalue for $K$, where the minimum degree of the graph is at least two. Note that some graphs have multiple defective eigenvalues and appear more than once.}
    \label{tab:comp data}
\end{table}

It is interesting to note that of the graphs listed in Table~\ref{tab:comp data}, only one graph has a Jordan block of size larger than two (see Figure~\ref{fig:10 diamonds conclusion}). Defective eigenvalues of $\pm 1$ are the cycle graphs of that order, as explained by Theorem~\ref{thm:unicyclic}. Six different graphs are counted in Table~\ref{tab:comp data} twice: cycle graphs $C_8$ and $C_{10}$ (eigenvalues $\pm 1$) and four graphs on 10 vertices (eigenvalues: $\pm \sqrt{2}i$ and $(-1\pm\sqrt{7}i)/2$). 

Finally, with regards to the constructions discussed in this section, we note the following:
\begin{itemize}
\item Of the graphs with defective eigenvalue $-2$, the graph on nine vertices and four of the five graphs on ten vertices are constructed from Bipartite Base Family.

\item Of the graphs with defective eigenvalues $\pm \sqrt{2}i$, 20 of the 22 graphs on eight vertices and 250 of the 324 graphs on nine vertices (160 from Graph (a), 90 from Graph (b)) are constructed from the Crustacean Family.

\item Of the graphs with defective eigenvalues $(-1\pm\sqrt{7}i)/2$, both graphs on seven vertices are constructed by Restricted Diamonds Family, as well as two of the sixteen graphs on eight and ten of the 156 graphs on nine vertices.
\end{itemize}

SageMath code used to generate these results is available upon request. 

\section{Conclusion}
In this paper, we showed that exploring the Jordan form of the non-backtracking matrix is equivalent to exploring the Jordan form of the matrix $K$ for graphs with at least two cycles. Moreover, for other graphs, the differences between $\nonback$ and $K$ can be easily qualified. As such, we can reduce the often larger matrix $\nonback$ into a more tractable matrix $K$. We also have built algebraic techniques to find Jordan chains and constructed three graph families which will have defective eigenvalues.

A future direction we propose is considering the diamonds graph shown in Figure~\ref{fig:10 diamonds conclusion}a. 
\begin{figure}[h!]
    \centering
    \begin{tabular}{ccc}
       \begin{tikzpicture}
    \node (l) at (0,0) {};

           \node[vertex] (v1) at ($(l)+(0,0)$) {};
           \node[vertex] (v2) at ($(l)+(0,-2)$) {};
           \node[vertex] (v3) at ($(l)+(-1,-1)$) {};
           \node[vertex] (v4) at ($(l)+(1,-1)$) {};
           \node[vertex] (v5) at ($(l)+(0,2)$) {};
           \node[vertex] (v6) at ($(l)+(-1,1)$) {};
           \node[vertex] (v7) at ($(l)+(1,1)$) {};

           \draw[thick] (v2)--(v1)--(v3)--(v2)--(v4)--(v1)--(v6)--(v5)--(v7)--(v6);
           \draw[thick] (v7)--(v1);

       \end{tikzpicture}  & \hspace{.1 in} &  \begin{tikzpicture}
    \node (r) at (0,0) {};
          
           \node[vertex] (v1) at ($(r)+(0,0)$) {};
           \node[vertex] (v2) at ($(r)+(0,-2)$) {};
           \node[vertex] (v3) at ($(r)+(-1,-1)$) {};
           \node[vertex] (v4) at ($(r)+(1,-1)$) {};
           \node[vertex] (v5) at ($(r)+(0,2)$) {};
           \node[vertex] (v6) at ($(r)+(-1,1)$) {};
           \node[vertex] (v7) at ($(r)+(1,1)$) {};
           \node[vertex] (u1) at ($(r)+(2,1)$) {};
           \node[vertex] (u2) at ($(r)+(2,-1)$) {};
           \node[vertex] (u3) at ($(r)+(3,0)$) {};
           \draw[thick] (v2)--(v1)--(v3)--(v2)--(v4)--(v1)--(v6)--(v5)--(v7)--(v6);
           \draw[thick] (v7)--(v1);
           \draw[thick] (v7)--(u1)--(u3)--(u2)--(v4)--(u1);
           \draw[thick] (v7) to [out=-90,in=0] (v3);
       \end{tikzpicture}\\
       (a)  & & (b)
    \end{tabular}
   
    \caption{(a) The diamonds graph and (b) a graph on $10$ vertices with $\lambda= \frac{-1 \pm \sqrt{7}i}{2}$ as a defective eigenvalue with Jordan block size of $3$. This is the only graph with such a block of order ten or less.}
    \label{fig:10 diamonds conclusion}
\end{figure}
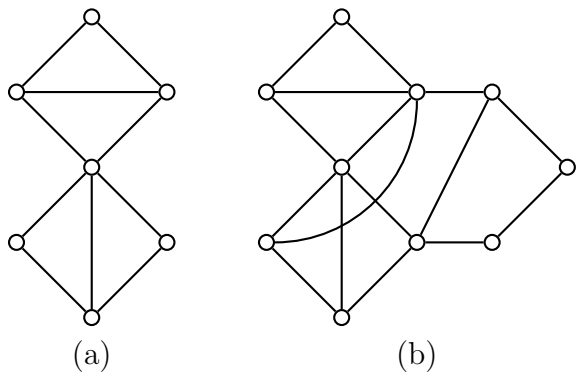
This graph on seven vertices appears as a subgraph in both defective graphs of order seven, and fifteen of the sixteen defective graphs on order eight. We also note that this subgraph appears in the only graph on ten or fewer vertices with a Jordan block of size three or more (see Figure~\ref{fig:10 diamonds conclusion}b). It is interesting to note that the diamonds graph is not itself defective, but does have $(-1\pm\sqrt{7} i)/2$ as eigenvalue (with algebraic and geometric multiplicity two) and the eigenvector in Proposition~\ref{prop:RestDiamond} is in the eigenspace. We are interested in a more comprehensive explanation of graphs with this diamonds subgraph, in particular graphs that include larger Jordan blocks.

\section*{Acknowledgements}

We would like to thank Xinyu Wu and Jane Breen for helpful discussions with regards to the infinite families of defective graphs in Section \ref{sec: fams}.

This project started at the American Institute of Mathematics workshop on Spectral Graph and Hypergraph Theory: Connections and Applications, which took place in December 2021 with support from the National Science Foundation and the Fry Foundation.

This material is also based upon work supported by the National Science Foundation under Grant No. DMS-1928930 and the National Security Agency under Grant No. H98230-23-1-0004 while the authors participated in a program hosted by the Simons Laufer Mathematical Sciences Institute (formerly Mathematical Sciences
Research Institute) in Berkeley, California, during the summer of 2023.

\end{document}